\documentclass[reqno, 12pt]{amsart}
\pdfoutput=1
\makeatletter
\let\origsection=\section \def\section{\@ifstar{\origsection*}{\mysection}} 
\def\mysection{\@startsection{section}{1}\z@{.7\linespacing\@plus\linespacing}{.5\linespacing}{\normalfont\scshape\centering\S}}
\makeatother        
\usepackage{amsmath,amssymb,amsthm}    
\usepackage{mathabx}\changenotsign   
\usepackage{mathrsfs} 
\usepackage{dsfont} 
\usepackage[babel]{microtype}
\usepackage{xcolor}  	
\usepackage[backref]{hyperref}
\hypersetup{
	colorlinks,
    linkcolor={red!60!black},
    citecolor={green!60!black},
    urlcolor={blue!60!black},
}

\usepackage{bookmark}

\usepackage[abbrev,msc-links,backrefs]{amsrefs} 
\usepackage{doi}

\renewcommand{\PrintDOI}[1]{\doi{#1}}

\usepackage[T1]{fontenc}
\usepackage{lmodern}

\usepackage[english]{babel}
\numberwithin{equation}{section}

\usepackage{geometry}
\usepackage{enumitem}
\def\rmlabel{\upshape({\itshape \roman*\,})}

\def\alabel{\upshape({\itshape \alph*\,})}

\let\polishlcross=\l
\def\l{\ifmmode\ell\else\polishlcross\fi}

\def\qand{\quad\text{and}\quad}

\let\emptyset=\varnothing
\let\setminus=\smallsetminus

\makeatletter
\def\moverlay{\mathpalette\mov@rlay}
\def\mov@rlay#1#2{\leavevmode\vtop{   \baselineskip\z@skip \lineskiplimit-\maxdimen
   \ialign{\hfil$\m@th#1##$\hfil\cr#2\crcr}}}
\newcommand{\charfusion}[3][\mathord]{
    #1{\ifx#1\mathop\vphantom{#2}\fi
        \mathpalette\mov@rlay{#2\cr#3}
      }
    \ifx#1\mathop\expandafter\displaylimits\fi}
\makeatother

\DeclareFontFamily{U}  {MnSymbolC}{}
\DeclareSymbolFont{MnSyC}         {U}  {MnSymbolC}{m}{n}
\DeclareFontShape{U}{MnSymbolC}{m}{n}{
    <-6>  MnSymbolC5
   <6-7>  MnSymbolC6
   <7-8>  MnSymbolC7
   <8-9>  MnSymbolC8
   <9-10> MnSymbolC9
  <10-12> MnSymbolC10
  <12->   MnSymbolC12}{}
\DeclareMathSymbol{\powerset}{\mathord}{MnSyC}{180}

\usepackage{tikz}
\usetikzlibrary{decorations.markings}
\usetikzlibrary{calc,positioning,decorations.pathmorphing,decorations.pathreplacing}

\makeatletter
\def\namedlabel#1#2{\begingroup
    #2\def\@currentlabel{#2}\phantomsection\label{#1}\endgroup
}
\makeatother

\newtheorem{theorem}             {Theorem}[section]
\newtheorem{lemma}     	[theorem] {Lemma}        
\newtheorem{conjecture}	[theorem] {Conjecture}

\newtheorem{proposition}[theorem] {Proposition}   
\newtheorem{corollary}	[theorem] {Corollary}

\newtheorem{claim}	[theorem] {Claim}

\let\eps=\varepsilon
\let\theta=\vartheta
\let\rho=\varrho
\let\phi=\varphi

\def\NN{\mathds N}
\def\ZZ{\mathds Z} 

\def\RR{\mathds R}

\def\cI{\mathcal I}

\begin{document}

\title{On the local density problem for graphs of given odd-girth}

\author[W. Bedenknecht]{Wiebke Bedenknecht}
\address{Fachbereich Mathematik, Universit\"at Hamburg, Hamburg, Germany}
\email{Wiebke.Bedenknecht@uni-hamburg.de}
\email{Christian.Reiher@uni-hamburg.de}
\email{schacht@math.uni-hamburg.de}

\author[G. O. Mota]{Guilherme Oliveira Mota}
\address{Instituto de Matem\'atica e Estat\'{\i}stica, Universidade de
   S\~ao Paulo, S\~ao Paulo, Brazil}
\email{mota@ime.usp.br}

\author[Chr. Reiher]{Christian Reiher}
\author[M. Schacht]{Mathias Schacht}

\thanks{The second author was supported by FAPESP (Proc.~2013/11431-2, 2013/20733-2) and partially by CNPq (Proc.~459335/2014-6). The collaboration of the authors was supported by CAPES/DAAD PROBRAL (Proc. 430/15)
and by FAPESP (Proc.~2013/03447-6).}

\begin{abstract}
Erd\H{o}s conjectured that every $n$-vertex triangle-free graph contains a subset of $\lfloor n/2\rfloor$ vertices that spans at 
most $n^2/50$ edges. Extending a recent result of Norin and Yepremyan, we confirm this conjecture for graphs homomorphic to so-called Andr\'asfai graphs. As a consequence, Erd\H{o}s' conjecture holds for every triangle-free graph $G$ with minimum degree $\delta (G)>10n/29$ and if $\chi (G)\leq 3$ the degree condition can be relaxed to \mbox{$\delta (G)> n/3$}. In fact, we obtain a more general result for graphs of higher odd-girth.
\end{abstract}

\maketitle

\section{Introduction}

 We say an $n$-vertex graph $G$ is \emph{$(\alpha,\beta)$-dense} if every 
subset of $\lfloor \alpha n\rfloor$ vertices spans more than $\beta n^2$ edges. Given $\alpha\in (0,1]$, Erd\H{o}s, Faudree, Rousseau, and Schelp~\cite{ErFaRoSc94} asked for the 
minimum $\beta=\beta(\alpha)$ such that every $(\alpha,\beta)$-dense graph contains a triangle. For example, Mantel's theorem asserts that $\beta (1)=1/4$. More generally, Erd\H os et al.\ 
conjectured that for  $\alpha\geq 17/30$ the balanced complete bipartite graph 
gives the best lower bound for the function  $\beta(\alpha)$, which leads to
\begin{equation}\label{eq:1730}
	\beta (\alpha)=\frac{1}{4}(2\alpha -1)\,.
\end{equation}
 The same authors verified this conjecture for $\alpha \geq 0.648$ and the best result in this direction 
 is due to Krivelevich~\cite{Kr95}, who verified it
for every $\alpha \geq 3/5$.

We say a graph $G$ is a \emph{blow-up} of some graph~$F$, if 
there exists a graph homomorphism $\phi\colon G\to F$ with the property $\{x,y\}\in E(G)$ if and only if $\{\varphi(x),\varphi(y)\}\in E(F)$. Moreover, a blow-up is \emph{balanced} if the preimages~$\phi^{-1}(v)$ of all vertices $v\in V(F)$ 
have the same size.
For $\alpha<17/30$ balanced blow-ups
   of the 5-cycle yield a better lower bound
for $\beta(\alpha)$ and Erd\H os et al.\ conjectured
\begin{equation}\label{eq:53120}
	\beta (\alpha)=\frac{1}{25}(5\alpha -2)
\end{equation}
for $\alpha\in[53/120,17/30]$. For $\alpha<53/120$ it is known 
that balanced blow-ups of the Andr\'asfai graph $F_3$ (see Figure~\ref{figure:Fd}) lead to a better bound.
The special case $\beta(1/2)=1/50$
was considered before by Erd\H{o}s~\cite{Er76a} 
(see also~\cite{Er97a} for a monetary bounty for this problem).

\begin{conjecture}[Erd\H os]\label{conjecture:erdos}
Every $(1/2,1/50)$-dense graph contains a triangle.
\end{conjecture}

Currently, the best known upper bound on $\beta(1/2)$ is $1/36$ and was obtained by Krivelevich~\cite{Kr95}.
Besides the balanced blow-up of the $5$-cycle, Simonovits (see, e.g.,~\cite{Er97a}) noted that balanced blow-ups of the 
Petersen graph yield the same lower bound for Conjecture~\ref{conjecture:erdos} and, more generally, for~\eqref{eq:53120} in the corresponding range. 

Conjecture~\ref{conjecture:erdos} asserts that every triangle-free $n$-vertex graph~$G$
contains a subset of~$\lfloor n/2\rfloor$ vertices
that spans at most $n^2/50$ edges. Our first result (see Theorem~\ref{thm:main} below) 
verifies this for graphs~$G$ that are homomorpic to a triangle-free graph from a special class.

\subsection{Andr\'asfai graphs}
A well studied family of triangle-free graphs, which appear in the lower bound constructions for the function $\beta(\alpha)$ above, are the so-called \emph{Andr\'asfai graphs}. Those graphs already appeared in the work of Erd\H os~\cite{Er57} and Andr\'asfai~\cites{A62,A64}. Andr\'asfai graphs are contained in several interesting graph classes. For example, they are special 
circulant graphs, they appear as finite subgraphs of 2-dimensional Borsuk graphs, and can be described as the  complement of an appropriate power of a cycle. For our purposes it is convenient to define them as follows.

 For an integer $d\geq 1$
the Andr\'asfai graph $F_d$ is the $d$-regular 
graph with vertex set
\[
	V(F_d)=\{v_1,\ldots,v_{3d-1}\}\,,
\] 
where $\{v_i,v_j\}$ forms an edge if
\begin{equation}\label{eq:def:Fd}
	d\leq |i-j| \leq 2d-1\,.
\end{equation}
Note that $F_1=K_2$ and $F_2=C_5$ (see Figure~\ref{figure:Fd}).
It is easy to check that Andr\'asfai graphs are 
triangle-free and balanced blow-ups of these graphs play a prominent r\^ole 
in connection with extremal problems for triangle-free graphs 
(see, e.g.,~\cites{Er57,A62,A64,AnErSo74,Ha82,Ji95,ChJiKo97,W73}). 

\begin{figure}[ht]
\begin{tikzpicture}[scale=0.93]
	\draw[opacity=0] (-2.6,-2.6) rectangle (2.6,2.6);
\foreach \x in {1,2,...,5}{
 		\fill (162-\x*72:2) circle (2.8pt);
 		\node at (162-\x*72:2.4) {$v_\x$};
		\draw[thick] (162-\x*72:2) -- (18-\x*72:2);
	}
\end{tikzpicture}	
\hfill
\begin{tikzpicture}[scale=0.93]
	\draw[opacity=0] (-2.6,-2.6) rectangle (2.6,2.6);
\foreach \x in {1,2,...,8}{
 		\fill (135-\x*45:2) circle (2.8pt);
		\node at (135-\x*45:2.4) {$v_\x$};
		\draw[thick] (135-\x*45:2) -- (-\x*45:2);
	} 
	\foreach \x in {1,2,3,4}
		\draw[thick] (135-\x*45:2) -- (-45-\x*45:2);
\end{tikzpicture}
\hfill
\begin{tikzpicture}[scale=0.93]
	\draw[opacity=0] (-2.6,-2.6) rectangle (2.6,2.6);
\foreach \x in {1,2,...,11}{
		\fill (122.727-\x*32.727:2) circle (2.8pt);
		\node at (122.727-\x*32.727:2.4) {$v_{\x}$};
		\draw[thick] (122.727-\x*32.727:2) -- (-8.181-\x*32.727:2);
		\draw[thick] (122.727-\x*32.727:2) -- (-40.908-\x*32.727:2);
	}
\end{tikzpicture}
\caption{Andr\'asfai graphs $F_2$, $F_3$, and $F_4$.}
\label{figure:Fd}
\end{figure}
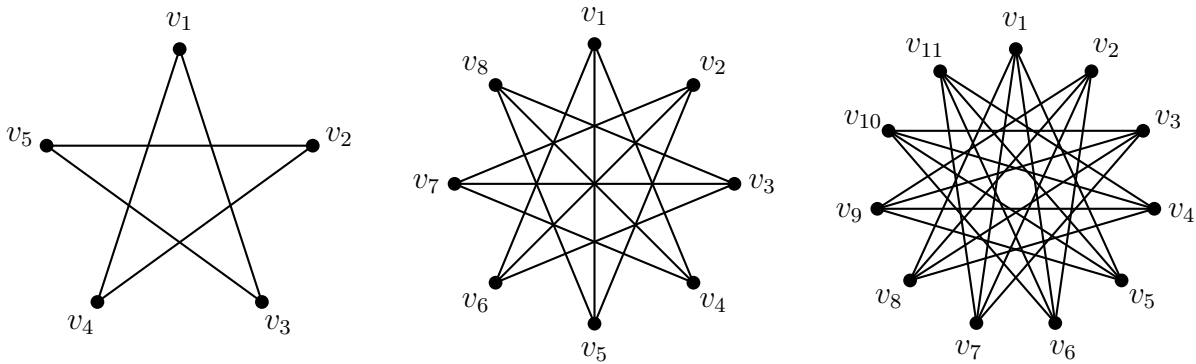

Our first result validates Conjecture~\ref{conjecture:erdos}
(stated in the contrapositive) for graphs homomorphic to some Andr\'asfai graph. 
\begin{theorem}\label{thm:main}
If  a graph $G$ is homomorphic to an Andr\'asfai graph $F_d$ for some integer~\mbox{$d\geq 1$}, then $G$ is not $(1/2,1/50)$-dense.
\end{theorem}
Since $F_d$ is homomorphic to $F_{d'}$ if and only if $d'\geq d$, 
Theorem~\ref{thm:main} extends recent work of Norin and Yepremyan~\cite{NoYe15}, who obtained such a result
for $n$-vertex graphs~$G$ homomorphic to~$F_5$ with the additional minimum degree assumption $\delta(G)\geq 5n/14$.

Owing to the work of Chen, Jin, and Koh~\cite{ChJiKo97}, which asserts that every triangle-free 
$3$-chromatic $n$-vertex graph $G$ with minimum degree $\delta(G)>n/3$ is homomorphic to some Andr\'asfai graph,
we deduce from Theorem~\ref{thm:main} that Conjecture~\ref{conjecture:erdos} holds for all such graphs~$G$.
Similarly, combining Theorem~\ref{thm:main} with a result of Jin~\cite{Ji95}, which
asserts that triangle-free graphs~$G$ with 
$\delta(G)>10n/29$ are homomorphic 
to~$F_9$, implies Conjecture~\ref{conjecture:erdos} for those graphs as well. 
We summarise these direct consequences of Theorem~\ref{thm:main} in the following corollary.
\begin{corollary}
	Let $G$ be a triangle-free graph on $n$ vertices.
	\begin{enumerate}[label=\alabel]
		\item\label{it:cora} If $\delta(G)> 10n/29$, then $G$ is not $(1/2,1/50)$-dense.
		\item\label{it:corb}If $\delta(G)> n/3$ and $\chi(G)\leq 3$, then $G$ is not $(1/2,1/50)$-dense.
	\end{enumerate}
\end{corollary} 
We remark that part~\ref{it:cora} slightly improves earlier results of Krivelevich~\cite{Kr95} and of 
Norin and Yepremyan~\cite{NoYe15} (see also~\cite{KeSu06} where an average degree condition was considered).

\subsection{Generalised Andr\'asfai graphs of higher odd-girth}
We consider the following straightforward variation of Andr\'asfai graphs of \emph{odd-girth} 
at least $2k+1$, i.e., graphs without odd cycles of length at most $2k-1$.
For integers $k\geq 2$ and $d\geq 1$
let $F^k_d$ be the $d$-regular 
graph with vertex set
\[
	V(F^k_d)=\{v_1,\dots,v_{(2k-1)(d-1)+2}\}\,,
\] 
where $\{v_i,v_j\}$ forms an edge if
\begin{equation}\label{eq:def:Fkd}
	(k-1)(d-1)+1\leq |i-j| \leq k(d-1)+1\,.
\end{equation}
In particular, for $k=2$ we recover the definition of the Andr\'asfai graphs from~\eqref{eq:def:Fd} and
for general $k\geq 2$ we have $F^k_1=K_2$, $F^k_2=C_{2k+1}$ and for every $d\geq 2$ the graph~$F^k_d$ 
has odd-girth~$2k+1$ (see Figure~\ref{figure:Fkd}).

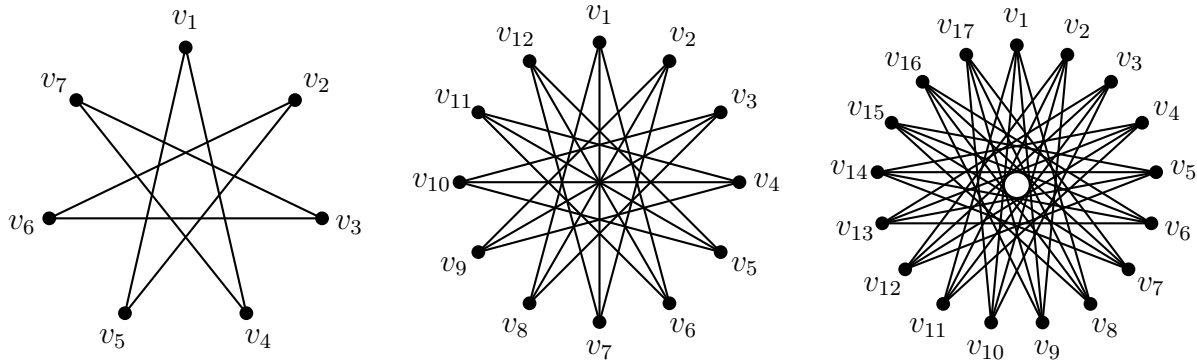
\begin{figure}[ht]
\begin{tikzpicture}[scale=0.93]
	\draw[opacity=0] (-2.6,-2.6) rectangle (2.6,2.6);
\foreach \x in {1,2,...,7}{
 		\fill (141.4286-\x*51.4286:2) circle (2.8pt);
 		\node at (141.4286-\x*51.4286:2.4) {$v_\x$};
		\draw[thick] (141.4286-\x*51.4286:2) -- (-12.8572-\x*51.4286:2);
	}
\end{tikzpicture}	
\hfill
\begin{tikzpicture}[scale=0.93]
	\draw[opacity=0] (-2.6,-2.6) rectangle (2.6,2.6);
\foreach \x in {1,2,...,12}{
 		\fill (120-\x*30:2) circle (2.8pt);
		\node at (120-\x*30:2.4) {$v_{\x}$};
		\draw[thick] (120-\x*30:2) -- (-30-\x*30:2);
	} 
	\foreach \x in {1,2,...,6}
		\draw[thick] (120-\x*30:2) -- (-60-\x*30:2);
\end{tikzpicture}
\hfill
\begin{tikzpicture}[scale=0.93]
	\draw[opacity=0] (-2.6,-2.6) rectangle (2.6,2.6);
\foreach \x in {1,2,...,17}{
		\fill (111.1765-\x*21.1765:2) circle (2.8pt);
		\node at (111.1765-\x*21.1765:2.4) {$v_{\x}$};
		\draw[thick] (111.1765-\x*21.1765:2) -- (-37.0588-\x*21.1765:2);
		\draw[thick] (111.1765-\x*21.1765:2) -- (-58.2353-\x*21.1765:2);
	}
\end{tikzpicture}
\caption{Generalised Andr\'asfai graphs $F^3_2$, $F^3_3$, and $F^3_4$ of odd-girth 7.}
\label{figure:Fkd}
\end{figure}

Our main result generalises Theorem~\ref{thm:main} 
for graphs of odd-girth at least $2k+1$. 
In fact, the constant $\frac{1}{2(2k+1)^2}$ appearing in Theorem~\ref{thm:main2}
is best possible as balanced blow-ups of~$C_{2k+1}$ show. One can attain this bound by taking $k$ mutually independent parts of the blow-up completely and the missing $n/2-kn/(2k+1)$ vertices from one part, which has only edges to one of the parts that we already chose.

\begin{theorem}\label{thm:main2}
	If a graph $G$ is homomorphic to a generalised Andr\'asfai graph $F^k_d$ for some integers $k\geq 2$ and $d\geq 1$, 
	then $G$ is not $(\frac{1}{2},\frac{1}{2(2k+1)^2})$-dense.
\end{theorem}

Analogous to the relation between Conjecture~\ref{conjecture:erdos} and Theorem~\ref{thm:main} one may wonder
if every $(\frac{1}{2},\frac{1}{2(2k+1)^2})$-dense graph contains an odd cycle of length at most~$2k-1$. 
Letzter and Snyder \cite{LS} showed that a  graph $G$ on $n$ vertices with $\delta(G)>\frac{n}{5}$ and odd-girth at least 7 is homomorphic to $F_k^3$, for some $k$.
Therefore combining this result with Theorem~\ref{thm:main2} we get the following.
\begin{corollary}
	Let $G$ be a graph with odd-girth at least $7$ on $n$ vertices.
	If $\delta(G)>\frac{n}{5}$, then~$G$ is not $(\frac{1}{2},\frac{1}{98})$-dense.
\end{corollary}

A similar question for even holes is not interesting, because every dense graph contains a $4$-cycle.

For $k=2$ Theorem~\ref{thm:main2} reduces to Theorem~\ref{thm:main} and the rest of this work 
concerns the proof of Theorem~\ref{thm:main2}.
The proof is given in Section~\ref{sec:general} and makes 
use of a geometric representation of graphs homomorphic to generalised Andr\'asfai graphs, which 
we introduce in Section~\ref{sec:geoA}.

\section{A geometric characterisation of generalised Andr\'asfai graphs}\label{sec:geoA}
We consider graphs $G$ that are homomorphic to some generalised Andr\'asfai graph $F^k_d$.  
For the proof of Theorem~\ref{thm:main2}
it will be convenient to work with a geometric representation of such graphs~$G$.
In fact, already the original definitions of Andr\'afai graphs in~\cites{Er57,A62} were geometric. 
In that representation we will arrange the vertices of~$G$ on the unit circle~$\RR/\ZZ$ 
and edges between two vertices 
$x$ and~$y$ may only appear depending on their angle with respect to the centre of the circle (see Lemma~\ref{lemma:Fkd}). For the proof of Theorem~\ref{thm:main2} it suffices to 
consider edge maximal graphs~$G$ that are homomorphic to $F^k_d$ for some integers $k\geq 2$ and $d\geq 1$. 
In other words, we may assume $G$ is a blow-up of 
$F^k_d$.

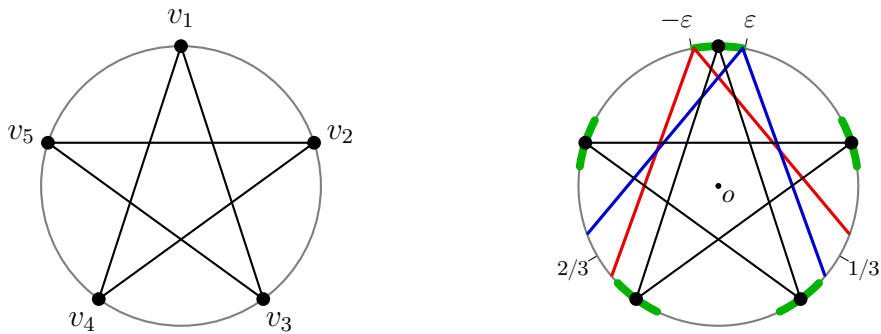
\begin{figure}[ht]
\begin{tikzpicture}[scale=0.93]
	\draw[opacity=0] (-2.6,-2.6) rectangle (2.6,2.6);

	\draw[thick,black!50!white] (0,0) circle (2.0);

 	\foreach \x in {1,2,...,5}{
 		\fill (162-\x*72:2) circle (2.8pt);
 		\node at (162-\x*72:2.4) {$v_\x$};
		\draw[thick] (162-\x*72:2) -- (18-\x*72:2);
	}
\end{tikzpicture}	
\hspace{2cm}
\begin{tikzpicture}[scale=0.93]
	\draw[opacity=0] (-2.6,-2.6) rectangle (2.6,2.6);

	\draw (101:2) -- (101:2.15);
	\draw (79:2) -- (79:2.15);
	\node at ($(101:2.4)-(4pt,-0.3pt)$) {\footnotesize$-\eps$};
	\node at (79:2.4) {\footnotesize$\eps$};
	
	\draw (-30:2) -- (-30:2.15);
	\draw (210:2) -- (210:2.15);
	\node at (-30:2.4) {\tiny$1/3$};
	\node at (210:2.4) {\tiny$2/3$};
	
	\foreach \i in { 28,  100,  172,  244, 316}
    	\draw[thick,black!50!white] (\i:2cm) arc [start angle=\i,radius=2cm, delta angle=52];
	
    \foreach \i in {8,  80, 152, 224, 296}     
    	\draw[very thick, line width = 3pt, 
green!70!black, line cap=round] 
			(\i:2cm) arc [start angle=\i,radius=2cm, delta angle=20];

	\draw[red!90!black, very thick, line cap=round] (28+72:2) -- (130+18+72:1.985);
	\draw[red!90!black, very thick, line cap=round] (28+72:2) -- (250+18+72:1.985);
	\draw[blue!80!black, very thick, line cap=round] (8+72:2) -- (110+18+72:1.985);
	\draw[blue!80!black, very thick, line cap=round] (8+72:2) -- (230+18+72:1.985);

	\foreach \x in {1,2,...,5}
		\draw[black,thick] (162-\x*72:2) -- (18-\x*72:2);

 	\foreach \x in {1,2,...,5}{
\fill[black] (162-\x*72:2) circle (2.8pt);
	}
	
	\fill[black] (0,0) circle (1.2pt);
	\node at (-45:0.2) {\footnotesize$o$};

\end{tikzpicture}
\caption{A copy of $F_2=C_5$ and a representation of a blow-up on the unit circle.}
\label{fig:C5-circle}
\end{figure}

For example, let $G$ be a blow-up of $F_2=C_5$.
One can distribute the vertices of $F_2$ equally spaced on the unit circle (see Figure~\ref{fig:C5-circle}).
Then we place all vertices of~$G$ that correspond to the blow-up class of~$v_i$ into a small $\eps$-ball 
around $v_i$ on the unit circle (cf.\ green arcs in Figure~\ref{fig:C5-circle}).
For a sufficiently small $\eps$, all vertices in an $\eps$-ball around~$v_i$ have the same neighbours
and they can be characterised by having their smaller angle with respect to the centre bigger than $120^\circ$
(cf.\ red and blue lines in Figure~\ref{fig:C5-circle}).
The following lemma states this fact for blow-ups of generalised Andr\'asfai graphs.

\begin{lemma}\label{lemma:Fkd}
If $G$ is a blow-up of a generalised Andr\'asfai graph $F^k_d$ 
for some integers $k\geq 2$ and $d\geq 1$, then the vertices of $G$ can be arranged on the unit
circle $\RR/\ZZ$ with centre~$o$ such that 
\begin{equation}\label{eq:xoy}
	\{x,y\}\in E(G)
	\qquad\Longleftrightarrow\qquad
	\sphericalangle xoy>\frac{k-1}{2k-1}\cdot360^\circ\,,
\end{equation}
where $\sphericalangle xoy$ denotes the smaller angle between $x$ and $y$ with respect to~$o$.
\end{lemma}
We remark that conversely every graph $G=(V,E)$ with $V\subseteq \RR/\ZZ$ satisfying~\eqref{eq:xoy}
is a blow-up of $F^k_d$ for some appropriate $d\geq 1$. However, since this direction is not utilised here, 
we omit the formal proof of this observation.
\begin{proof}[Proof of Lemma~\ref{lemma:Fkd}]
For integers $k\geq 2$ and $d\geq 1$ let $G$ be a blow-up of the generalised
Andr\'asfai graph~$F^k_d$ (defined in~\eqref{eq:def:Fkd})
signified by some graph homomorphism $\phi\colon G\to F^k_d$
and let $m=(2k-1)(d-1)+2$ be the number of vertices of $F^k_d$.
Set
\[
	\eps=\frac{1}{2(2k-1)m}\,.
\]
For $i\in[m]$ we arrange the vertices of $G$ that are contained 
in $\phi^{-1}(v_i)$ in the $\eps$-ball around the point~$\tfrac{i-1}{m}$.
Owing to the symmetry it suffices to check that~\eqref{eq:xoy} holds 
for an arbitrary  vertex $x\in\phi^{-1}(v_1)\subseteq V(G)$.

By definition of~$F^k_d$ the neighbourhood of $v_1$ is 
\[
	N(v_1)=\{v_{(k-1)(d-1)+2},\dots,v_{k(d-1)+2}\}\,.
\]
Note that the choice of $\eps$ gives
\[
	\big(\tfrac{(k-1)(d-1)+1}{m}-\varepsilon, \tfrac{k(d-1)+1}{m}+\varepsilon\big)
	= 
	\big(\tfrac{k-1}{2k-1}+\varepsilon,\tfrac{k}{2k-1}-\varepsilon\big)
\]
and, consequently, all neighbours~$y$ of $x$ are placed in the interval $(\tfrac{k-1}{2k-1}+\varepsilon,\tfrac{k}{2k-1}-\varepsilon)$.
Since $x\in\phi^{-1}(v_1)$ itself is placed in $(-\eps,\eps)$, this implies the forward direction of~\eqref{eq:xoy}.
The converse direction follows 
from the observation 
\[
	\big(\tfrac{i-1}{m}-\varepsilon, \tfrac{i-1}{m}+\varepsilon\big)
	\cap
	\big(\tfrac{k-1}{2k-1}-\varepsilon,\tfrac{k}{2k-1}+\varepsilon\big)
	=\emptyset
\]
for every $i\in[m]\setminus\{(k-1)(d-1)+2,\dots,k(d-1)+2\}$.
\end{proof}

We close this section with a few useful estimates on the number of vertices contained in
intervals of $\RR/\ZZ$ for geometric representations of blow-ups~$G$ of generalised 
Andr\'asfai graphs. Let $V$ be the set of points of the unit circle that are identified 
with the vertices of~$G$. For an interval $I\subseteq\RR/\ZZ$, 
we write $\lambda(I)$ for the number of vertices of $G$ contained in~$I$, i.e.,
\begin{equation}\label{eq:def:lambda}
	\lambda(I)=|V\cap I|\,.
\end{equation}
This defines expressions such as $\lambda\bigl([a, b]\bigr)$, $\lambda\bigl([a, b)\bigr)$, etc.

Since subsets of $\lfloor  n/2\rfloor$ vertices are of special interest, we denote
for every~$\xi\in \RR/\ZZ$ by~$z_\xi$ the vertex from~$V$ with the property
\begin{equation}\label{eq:def:zxi}
	\lambda\bigl([\xi,z_\xi]\bigr)=\lfloor  n/2\rfloor\,.
\end{equation}
In the proof of Theorem~\ref{thm:main2} we shall use the following lemma and, since the proof 
will be carried out by contradiction, the graphs $G$ that we shall consider also satisfy the density assumption 
for parts~\ref{it:u:5} and~\ref{it:u:6}.

\begin{lemma}\label{lem:useful}
	For integers $k\geq 2$ and $d\geq 1$ let $G=(V,E)$ 
	be a blow-up of the generalised Andr\'asfai graph $F^k_d$ 
	having a geometric representation
	with $V\subseteq \RR/\ZZ$ satisfying~\eqref{eq:xoy} and~$|V|=n$. Then the following holds 
	for every interval $I\subseteq\RR/\ZZ$:
	\begin{enumerate}[label=\rmlabel]
		\item\label{it:u:1} If $I$ has length at most $\frac{k-1}{2k-1}$, then $V\cap I$
		is an independent set in $G$ and $\lambda(I)\leq \alpha(G)$.
		\item\label{it:u:2} If $I$ has length at most $\frac{1}{2k-1}$, then 
		$\lambda(I)\leq (2k-3)\alpha(G)-(k-2)n$.
		\item\label{it:u:3} If $I$ has length at least $\frac{1}{2k-1}$, then 
		$\lambda(I)\geq n-2\alpha(G)$.
	\end{enumerate}
	If in addition~$G$ is $(\frac{1}{2},\frac{1}{2(2k+1)^2})$-dense and $2(2k+1)\mid n$, then 
 	the following holds for $\xi\in\RR/\ZZ$:
	\begin{enumerate}[label=\rmlabel]
	\setcounter{enumi}{3}

		\item\label{it:u:5} 
		If $\lambda\big([\xi,\xi+\frac{k-1}{2k-1}]\big)=\alpha(G)$, then $\lambda\big([\xi,z_\xi-\tfrac{k-1}{2k-1})\big)> 2\alpha(G)-\frac{2k-1}{2k+1}n$.
\item\label{it:u:6} We have
			$\lambda\big((\xi-\tfrac{1}{2k-1},\xi+\tfrac{1}{2k-1})\big)
					> \frac{4}{2k+1}n-2\lambda\big((\xi+\tfrac{k-1}{2k-1},\xi-\tfrac{k-1}{2k-1})\big)$.
\end{enumerate}
\end{lemma}
\begin{proof}
Part~\ref{it:u:1} follows directly from the definition of the geometric 
representation in~\eqref{eq:xoy}. For part~\ref{it:u:2} we note 
that 
\[
	(2k-3)\frac{k-1}{2k-1}=(k-2)+\frac{1}{2k-1}\,.
\]
Consequently, there exist $2k-3$ consecutive intervals of length $\frac{k-1}{2k-1}$ that
wrap $k-2$ times around~$\RR/\ZZ$ in such a way that only $I$ is covered $k-1$ times.
Therefore,~\ref{it:u:1} yields
\[
	(2k-3)\alpha(G)\geq (k-2)n+\lambda(I)
\]
and the desired estimate follows.

Part~\ref{it:u:3} is also a consequence of~\ref{it:u:1} and the observation that
there are two intervals of length at most $\frac{k-1}{2k-1}$ that together with $I$
cover $\RR/\ZZ$ once.

In the proofs of parts~\ref{it:u:5} and~\ref{it:u:6} we make use of the inequality
\begin{equation}\label{eq:u:4}
	\lambda\big([\xi,z_\xi-\tfrac{k-1}{2k-1})\big)
	>
	\frac{2n}{2k+1}-2\lambda\big((\xi+\tfrac{k-1}{2k-1},z_\xi]\big)\,,
\end{equation}
which we show first. For that we note that~\eqref{eq:xoy} implies
\begin{align*}
	e_G\big([\xi,z_\xi]\cap V\big)
	&\leq 
	\lambda\big([\xi,z_\xi-\tfrac{k-1}{2k-1})\big)\cdot \lambda\big((\xi+\tfrac{k-1}{2k-1},z_\xi]\big)\,.
\end{align*}
Hence, the additional assumption that~$G$ is $(\frac{1}{2},\frac{1}{2(2k+1)^2})$-dense
combined with the simplest case of the 
inequality between the arithmetic and geometric mean yields
\[
	\left(\frac{n}{2k+1}\right)^2
	<
	2e_G\big([\xi,z_\xi]\cap V\big)
	\leq
	\frac{1}{4}\Big(\lambda\big([\xi,z_\xi-\tfrac{k-1}{2k-1})\big)+2\lambda\big((\xi+\tfrac{k-1}{2k-1},z_\xi]\big)\Big)^2
	\,,
\]
which establishes~\eqref{eq:u:4}.

The remaining parts~\ref{it:u:5} and~\ref{it:u:6} follow from~\eqref{eq:u:4}.
In fact, for~\ref{it:u:5}  the additional assumption $\lambda\big([\xi,\xi+\frac{k-1}{2k-1}]\big)=\alpha(G)$
yields $\lambda\big((\xi+\tfrac{k-1}{2k-1},z_\xi]\big)=n/2-\alpha(G)$ and, hence,~\ref{it:u:5}
follows from~\eqref{eq:u:4}.

For the proof of~\ref{it:u:6} we will apply~\eqref{eq:u:4} twice. First we apply it for the given 
$\xi\in\RR/\ZZ$ and, since by~\ref{it:u:1} we also have 
$z_\xi\in(\xi+\tfrac{k-1}{2k-1},\xi+\tfrac{k}{2k-1})$, we obtain
\begin{equation}\label{eq:u:6a}
	\lambda\big([\xi,\xi+\tfrac{1}{2k-1})\big)
	\geq
	\lambda\big([\xi,z_\xi-\tfrac{k-1}{2k-1})\big)
	\overset{\eqref{eq:u:4}}{>}
	\frac{2n}{2k+1}-2\lambda\big((\xi+\tfrac{k-1}{2k-1},z_\xi]\big)\,.
\end{equation}
The second symmetric application of~\eqref{eq:u:4} in $-\RR/\ZZ$ 
to~$-\xi$ yields
\begin{equation}\label{eq:u:6b}
\lambda\big((\xi-\tfrac{1}{2k-1},\xi]\big)
	\overset{\eqref{eq:u:4}}{>}
	\frac{2n}{2k+1}-2\lambda\big([z'_{\xi},\xi-\tfrac{k-1}{2k-1})\big)\,,
\end{equation}
for $z'_\xi\in(z_\xi,\xi)$ with $\lambda\big([z'_\xi,\xi]\big)= n/2$.
Consequently, if $\xi\not\in V$ then summing the inequalities~\eqref{eq:u:6a} and~\eqref{eq:u:6b} 
yields part~\ref{it:u:6}. However, if $\xi\in V$ then still the same conclusion follows, 
since $(2k+1)\mid n$ implies that the right-hand sides of~\eqref{eq:u:6a} and~\eqref{eq:u:6b} are integers
and both inequalities are strict.
\end{proof}

\section{Blow-ups of generalised Andr\'asfai graphs}\label{sec:general}
In this section we establish Theorem~\ref{thm:main2}. For that it suffices to show 
that blow-ups~$G$ of generalised Andr\'asfai graphs~$F^k_d$ are not $(\frac{1}{2},\frac{1}{2(2k+1)^2})$-dense 
and we will appeal to the geometric representation 
from Lemma~\ref{lemma:Fkd} of such graphs.
The strategy of our proofs is that we try to find an interval of consecutive vertices spanning few edges. To this end we distinguish two cases depending on the independence number $\alpha(G)$ and start with the case
that $\alpha(G)$ is not too large.

\begin{proposition}\label{thm:Fd1}
If $G$ is a blow-up of a generalised Andr\'asfai graph $F^k_d$ 
for some integers $k\geq 2$ and $d\geq 1$ with $|V(G)|=n$
and $\alpha(G)< \frac{k}{2k+1}n$, then $G$ is not $(\frac{1}{2},\frac{1}{2(2k+1)^2})$-dense.
\end{proposition}

\begin{proof}
Let $G$ be a blow-up of $F^k_d$ with $|V(G)|=n$ and $\alpha(G)< \frac{k}{2k+1}n$. 
Without loss of generality we may assume that $n$ is divisible by $2(2k+1)$.
This follows from the observation, that a graph~$G$ is $(\frac{1}{2},\frac{1}{2(2k+1)^2})$-dense 
if and only if the balanced blow-up of $G$ obtained by replacing each vertex by $2(2k+1)$ 
vertices is $(\frac{1}{2},\frac{1}{2(2k+1)^2})$-dense.

Suppose for the sake of contradiction that $G$ is $(\frac{1}{2},\frac{1}{2(2k+1)^2})$-dense.
From now on consider the geometric representation of $G$ given by Lemma~\ref{lemma:Fkd}.
Let $V$ be the set of points of the unit circle that are identified with the vertices of~$G$.
Recall that in~\eqref{eq:def:lambda} we defined~$\lambda(I)$ as the number of vertices 
contained in an interval $I\subseteq \RR/\ZZ$.
It will sometimes be convenient to count vertices on the boundary of an interval 
only with weight $1/2$. For that we write terms like 
$\lambda(\langle a,b\rangle)$, $\lambda(\langle a,b))$, where the 
brackets ``$\langle$'' or ``$\rangle$'' mark that the left or right end-point
of the respective interval is only counted $1/2$ if it is a vertex. 
Also recall that for $\xi\in \RR/\ZZ$ we defined $z_\xi\in V$ in~\eqref{eq:def:zxi}.
Since by our assumption $\alpha(G)<n/2$, we infer from part~\ref{it:u:1} of Lemma~\ref{lem:useful}
that
\begin{equation}\label{eq:zxi}
	z_\xi \in \big(\xi+\tfrac{k-1}{2k-1},\xi+\tfrac{k}{2k-1}\big)\,,
\end{equation}
which yields together with Lemma \ref{lem:useful}\ref{it:u:1} that
\begin{equation}\label{eq:zxi2}
	\sum_{x\in V\cap[\xi,z_\xi-\tfrac{k-1}{2k-1})}\big|N_G(x)\cap(x,z_\xi]\big|
	=
	e_G\bigl([\xi,z_\xi]\cap V\bigr)\,.
\end{equation}

Moreover, part~\ref{it:u:2} of Lemma~\ref{lem:useful} applied 
to intervals $[x+\frac{k-1}{2k-1},x+\frac{k}{2k-1}]$
combined with the assumption $\alpha(G)<\frac{k}{2k+1}n$ leads to 
\begin{multline*}
	\lambda \big(\langle x+\tfrac{k-1}{2k-1},x+\tfrac{k}{2k-1}\rangle \big)
	\leq
	\lambda \big([x+\tfrac{k-1}{2k-1},x+\tfrac{k}{2k-1}]\big)\\
	\leq
	(2k-3)\alpha(G)-(k-2)n
	<
	(2k-3)\frac{k}{2k+1}n-(k-2)n
	=
	\frac{2}{2k+1}n
\end{multline*}
for every vertex $x\in V$. Consequently,
\begin{align*}
	\sum_{x\in V}\Big (\lambda\big(\langle x-\tfrac{k-1}{2k-1},x\rangle \big)
	+\lambda\big(\langle x,x+\tfrac{k-1}{2k-1} \rangle\big)\Big)
	&\!=\!
	\sum_{x\in V}\Big(\lambda \big(\langle x,x+1 \rangle\big)
	-\lambda\big(\langle x+\tfrac{k-1}{2k-1},x+\tfrac{k}{2k-1}\rangle\big)\Big)\\
	&\!>
	n^2-\frac{2}{2k+1}n^2
	= 
	\frac{2k-1}{2k+1}n^2
\end{align*}
and  by symmetry we may assume that
\begin{equation}\label{eq:2a}
\sum_{x\in V}\lambda \big(\langle x,x+\tfrac{k-1}{2k-1}\rangle \big)>\frac{1}{2}\cdot\frac{2k-1}{2k+1}n^2 \, .
\end{equation}
In view of~\eqref{eq:2a} the following claim seems a bit surprising and, in fact, it will 
lead to the desired contradiction. For a simpler notation we set 
\begin{equation}\label{eq:def:Vxi}
	V_\xi=V\cap[\xi,z_\xi-\tfrac{k-1}{2k-1})
\end{equation}
for $\xi\in \RR/\ZZ$.

\begin{claim}\label{cl:1a}
For every $\xi\in \RR/\ZZ$ we have
\[
	\sum_{x\in V_\xi} 
	\Big(\lambda \big(\langle x,x+\tfrac{k-1}{2k-1}\rangle \big)-\frac{1}{2}\cdot\frac{2k-1}{2k+1}n\Big)
	< 0\,.
\] 
\end{claim}

\begin{proof}[Proof of Claim~\ref{cl:1a}]
Fix some $\xi\in\RR/\ZZ$. Since we assume that $G$ is $(\frac{1}{2},\frac{1}{2(2k+1)^2})$-dense, we have 
\[
	\sum_{x\in V_\xi} \lambda \big((x+\tfrac{k-1}{2k-1},z_\xi]\big)
	=
	\sum_{x\in V_\xi}\big|N_G(x)\cap(x,z_\xi]\big|
	\overset{\eqref{eq:zxi2}}{=}
	e_G\bigl([\xi,z_\xi]\cap V\bigr)
	>
	\frac{n^2}{2(2k+1)^2}\,.
\]
Therefore,
\begin{align}\label{eq:3a}
	\sum_{x\in V_\xi} \lambda \big(\langle x,x+\tfrac{k-1}{2k-1}]\big)
	&\overset{\phantom{\eqref{eq:def:zxi}}}{=}
	\sum_{x\in V_\xi}\Big(\lambda\big(\langle x,z_\xi]\big)
		-\lambda\big(( x+\tfrac{k-1}{2k-1},z_\xi]\big)\Big)\nonumber\\
	&\overset{\phantom{\eqref{eq:def:zxi}}}{<}
	\sum_{x\in V_\xi}\lambda \big(\langle x,z_\xi]\big)
		-\frac{n^2}{2(2k+1)^2}\nonumber\\
	&\overset{\phantom{\eqref{eq:def:zxi}}}{=}\sum_{x\in V_\xi}\Big(\lambda\big([\xi,z_\xi]\big) 
		- \lambda \big([\xi,x\rangle \big)\Big)-\frac{n^2}{2(2k+1)^2}\nonumber\\
	&\overset{\eqref{eq:def:zxi}}{=}|V_\xi|\cdot\frac{n}{2}-\frac{n^2}{2(2k+1)^2}-\sum_{x\in V_\xi}\lambda \big([\xi,x\rangle \big)\,.
\end{align}
We observe 
\begin{equation}\label{eq:4a}
 \sum_{x\in V_\xi} \lambda \big([\xi,x\rangle \big)
 =
 \sum_{i=1}^{|V_\xi|}(i-\tfrac 12)
 =
 \frac{|V_\xi|^2}{2}
\end{equation}
and combining \eqref{eq:3a} and \eqref{eq:4a} yields
\begin{align*}
	\sum_{x\in V_\xi} \Big(\lambda \big(\langle x,x+\tfrac{k-1}{2k-1}]\big)-\frac{1}{2}\cdot\frac{2k-1}{2k+1}n\Big)
	&<
	|V_\xi|\cdot \left(\frac{n}{2}-\frac{1}{2}\cdot\frac{2k-1}{2k+1}n\right) 
	 	- \frac{n^2}{2(2k+1)^2}- \frac{|V_\xi|^2}{2}\\
	&=
	-\frac{1}{2}\left(|V_\xi|-\frac{n}{2k+1}\right)^2\leq 0\,,
\end{align*}
which establishes the claim.
\end{proof}

Now set $V^*=\bigl\{\xi\in \RR/\ZZ \colon \xi+\tfrac{k-1}{2k-1}\in V\bigr\}$. 
Starting with an arbitrary $x(0)\in V^*$ we define recursively a sequence of 
members of $V^*$ by putting
\[
	 x(i+1) = z_{x(i)}-\tfrac{k-1}{2k-1}
\]	 
for every $i\in \NN$.
Since $V^*$ is finite, this sequence is eventually periodic and thus we could have chosen $x(0)$
such that $x(m)=x(0)$ holds for some $m\ge 2$. Let $w\in \NN$ denote the number of times we wind around 
the circle when reaching $x(m)$ from $x(0)$ by this construction.
By Claim \ref{cl:1a} we know that
\[
	\sum_{i=0}^{m-1}
	\sum_{x\in V_{x(i)}} 
	\Big(\lambda \big(\langle x,x+\tfrac{k-1}{2k-1}\rangle \big)-\frac{1}{2}\cdot\frac{2k-1}{2k+1}n\Big)
	< 0\,.
\]
On the other hand, \eqref{eq:2a} yields
\begin{align*}
	\sum_{i=0}^{m-1} \sum_{x\in V_{x(i)}}
		\Big(\lambda \big(\langle x,x+\tfrac{k-1}{2k-1}\rangle \big)-\frac{1}{2}\cdot\frac{2k-1}{2k+1}n\Big)
	&\overset{\eqref{eq:def:Vxi}}{=} 
	w\cdot\sum_{x\in V}
		\Big(\lambda \big(\langle x,x+\tfrac{k-1}{2k-1}\rangle \big)-\frac{1}{2}\cdot\frac{2k-1}{2k+1}n\Big) \\
	&\overset{\eqref{eq:2a}}{>}
	w\cdot\left(\frac{1}{2}\cdot\frac{2k-1}{2k+1}n^2-\frac{1}{2}\cdot\frac{2k-1}{2k+1}n^2\right)=0\, ,
\end{align*}
which is a contradiction and concludes the proof of Proposition~\ref{thm:Fd1}.
\end{proof}

It is left to consider the case when $G$ contains a large  independent set. 

\begin{proposition}\label{thm:general}
If $G$ is a blow-up of a generalised Andr\'asfai graph $F^k_d$ 
for some integers $k\geq 2$ and $d\geq 1$ with $|V(G)|=n$
and $\alpha(G)\geq \frac{k}{2k+1}n$, then $G$ is not $(\frac{1}{2},\frac{1}{2(2k+1)^2})$-dense.
\end{proposition}
\begin{proof}
Similarly as in the proof of Proposition~\ref{thm:Fd1} we 
consider the geometric representation of an $n$-vertex graph $G$ 
that is a blow-up of a generalised Andr\'asfai graph 
$F^k_d$ and identify the vertex set of $G$ with some set $V\subseteq \RR/\ZZ$ so that~\eqref{eq:xoy} 
holds.
Again we may assume without loss of generality that $n$ is divisible by $2(2k+1)$ and
we suppose for a contradiction that $G$ is $(\tfrac{1}{2},\tfrac{1}{2(2k+1)^2})$-dense.
In particular, $\alpha(G)<n/2$ and the additional assumptions for parts~\ref{it:u:5} 
and~\ref{it:u:6} of Lemma~\ref{lem:useful} are satisfied.

Observe that every independent set of~$G$ is contained in some interval of $\RR/\ZZ$ of
length~$\tfrac{k-1}{2k-1}$. Therefore, without loss of generality we may assume that~$[0,\tfrac{k-1}{2k-1}]$ 
contains a maximum independent set, i.e.,
\[
	\lambda\big([0,\tfrac{k-1}{2k-1}]\big)=\alpha(G)\geq \frac{k}{2k+1}n\,.
\]

Recall that in~\eqref{eq:def:zxi} we defined a point $z_0$ with $\lambda\big([0, z_0]\big)=n/2$. 
Let the vertex $z'$ be defined similarly by $\lambda\bigl([z', \tfrac{k-1}{2k-1}]\bigr)=n/2$.
Then we have 
\begin{align*}
	\lambda\big((z_0, z'\,)\big) 
	&= 
	n-\lambda\big([z',0)\big)-\lambda\big([0,\tfrac{k-1}{2k-1}]\big)-\lambda\big((\tfrac{k-1}{2k-1},z_0]\big)\\
	&=
	n - (n/2-\alpha(G))-\alpha(G)-(n/2-\alpha(G))\\
	&=
	\alpha(G)
\end{align*}
and since $z_0$, $z'\in V$ the maximality of $\alpha(G)$
discloses that the interval $[z_0, z'\,]$ has at least the length $\tfrac{k-1}{2k-1}$.
Hence there is a closed subinterval $[b_k, b_0]$ of $[z_0, z'\,]$ whose length is exactly 
$\tfrac{k-1}{2k-1}$. We complete $b_0$ and $b_k$ to the vertices of a regular $(2k-1)$-gon, 
i.e., we consider the points $b_i\in \RR/\ZZ$ for $i\in\{0,\dots,2k-2\}$ 
such that the intervals $[b_i,b_{i+1}]$ have length $\tfrac{1}{2k-1}$ (see Figure~\ref{fig:proof}). 
Notice that $\alpha(G)<n/2$ entails
\begin{equation}\label{eq:gz0}
	z_0\in (b_{k-1},b_k]\,.
\end{equation}

\begin{figure}[ht]
\begin{tikzpicture}[scale=1.2]
	\draw[very thick,black!70!white] (0,0) circle (2.0);
	
	\draw[very thick, line width = 2pt, red!70!black] 
			(90:2) arc [start angle=90,radius=2, delta angle=-168];
	
\draw[thick, red!70!black] (90:1.9) -- (90:2.1);
	\draw[thick, red!70!black] (-78:1.9) -- (-78:2.1);
	\node at (90:2.28) {\small\textcolor{red!70!black}{$0$}};
	\node at (-78:2.35) {\small\textcolor{red!70!black}{$\frac{k-1}{2k-1}$}};

	\foreach \x in {0, 1, 2, 6, 7, 8, 9, 14}{
		\draw[thick] (101-\x*24:1.9) -- (101-\x*24:2.1);
	}
	\node at (103:2.3) {\small$b_{0}$};
	\node at (101-24:2.28) {\small$b_{1}$};
	\node at (101-48:2.28) {\small$b_{2}$};
	\node at (101-142:2.45) {\small$b_{k-2}$};
	\node at (101-163:2.4) {\small$b_{k-1}$};
	\node at (101-192:2.28) {\small$b_{k}$};
	\node at (101-216:2.28) {\small$b_{k+1}$};
	\node at (101+24:2.35) {\small$b_{2k-2}$};
	
	\draw[thick, green!50!black] (97:1.9) -- (97:2.1);
	\draw[thick, green!50!black] (-86:1.9) -- (-86:2.1);
	\node at (96:1.75) {\small\textcolor{green!50!black}{$z'$}};
	\node at (-86:1.73) {\small\textcolor{green!50!black}{$z_0$}};
	
\end{tikzpicture}
\caption{Largest independent set of $G$ is contained in the interval $[0,\frac{k-1}{2k-1}]$
	and the intervals $[0,z_0]$, $[z',\frac{k-1}{2k-1}]$ contain $n/2$ vertices each. 
	The $b_i$ form a regular $(2k-1)$-gon.}
\label{fig:proof}
\end{figure}
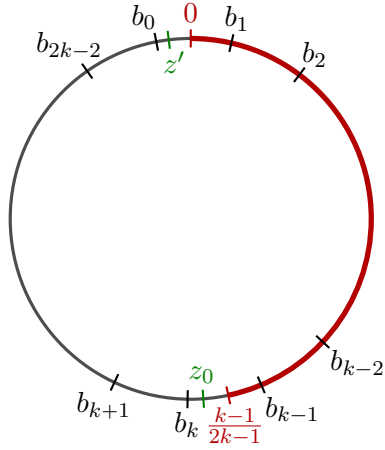

Below we apply Lemma~\ref{lem:useful}
to obtain several bounds on $\lambda\big([b_1,b_{k-1}]\big)$ and $\lambda\big([b_{k+1},b_{2k-2}]\big)$
that eventually lead to the desired contradiction. 
Applying Lemma~\ref{lem:useful}~\ref{it:u:5} with $\xi=0$ gives 
\[
	\lambda\big([0,b_1)\big)
	=
	\lambda\big([0,b_k-\tfrac{k-1}{2k-1})\big)
	\overset{\eqref{eq:gz0}}{\geq} 
	\lambda\big([0,z_0-\tfrac{k-1}{2k-1})\big)
	> 
	2\alpha(G)-\frac{2k-1}{2k+1}n
\]
and, by symmetry, we also have
\[
	\lambda\big((b_{k-1},\tfrac{k-1}{2k-1}]\big)
	> 
	2\alpha(G)-\frac{2k-1}{2k+1}n\,.
\]
Consequently, we arrive at
\begin{align}
	\lambda\big([b_{1},b_{k-1}]\big)
	&=
	\lambda\big([0,\tfrac{k-1}{2k-1}]\big)-\lambda\big([0,b_1)\big)-\lambda\big((b_{k-1},\tfrac{k-1}{2k-1}]\big)
	\nonumber\\
	&<
	\alpha(G)-2\left(2\alpha(G)-\frac{2k-1}{2k+1}n\right)
	=\frac{4k-2}{2k+1}n-3\alpha(G)\,.\label{eq:g1}
\end{align}
In particular, for the case $k=2$ this implies
\[
	0\leq \lambda\big([b_{1},b_{1}]\big)<\frac{6}{5}n-3\alpha(G)\,,
\]
which contradicts our assumption $\alpha(G)\geq2n/5$. Similarly, for $k=3$ 
inequality~\eqref{eq:g1} combined with Lemma~\ref{lem:useful}~\ref{it:u:3}
gives
\[
	n-2\alpha(G)
	\leq 
	\lambda\big([b_{1},b_{2}]\big)
	<
	\frac{10}{7}n-3\alpha(G)\,,
\]
which again contradicts the assumption $\alpha(G)\geq3n/7$ of this case.
Consequently, for the rest of the proof 
we can assume that $k\geq 4$.

Next we note that both intervals $(b_{k-1},b_{2k-2})$ and $(b_{k+1},b_{1})$ have length $\tfrac{k-1}{2k-1}$
and, hence, Lemma~\ref{lem:useful}~\ref{it:u:1} implies
\[
	\lambda\big((b_{k-1},b_{2k-2})\big)+\lambda\big((b_{k+1},b_{1})\big)\leq 2\alpha(G)
\]
and, therefore,
\begin{equation}\label{eq:g2}
	\lambda\big((b_{k+1},b_{2k-2})\big)
	\leq 
	2\alpha(G) - \lambda\big((b_{k-1},b_{1})\big)
	=
	2\alpha(G) - \Big(n - \lambda\big([b_{1},b_{k-1}]\big)\Big)\,.
\end{equation}

Finally, below we will verify 
\begin{equation}\label{eq:g3}
	\frac{4k-5}{2k+1}n-2\alpha(G)-\lambda\big([b_{1},b_{k-1}]\big)
	<
	\lambda\big((b_{k+1},b_{2k-2})\big)\,.
\end{equation}
Before we prove~\eqref{eq:g3}, we note that using~\eqref{eq:g2} as an upper bound for the right-hand side 
of~\eqref{eq:g3}
leads to
\[
	\frac{6k-4}{2k+1}n-4\alpha(G) 
	< 
	2\lambda\big([b_{1},b_{k-1}]\big)
	\overset{\eqref{eq:g1}}{<}
	\frac{8k-4}{2k+1}n-6\alpha(G)\,.
\]
This inequality contradicts the assumption $\alpha(G)\geq\frac{k}{2k+1}$ of the proposition and, hence, 
we conclude the proof by establishing~\eqref{eq:g3}.

For the proof of inequality~\eqref{eq:g3} we appeal to Lemma~\ref{lem:useful}~\ref{it:u:6} with 
$\xi=b_i$ for every $i=2,\dots,k-2$. We set
\[
	I_i=(b_i-\tfrac{1}{2k-1},b_i+\tfrac{1}{2k-1})=(b_{i-1},b_{i+1})
\]
and then in view of
\[
	(b_i+\tfrac{k-1}{2k-1},b_i-\tfrac{k-1}{2k-1})=(b_{i+k-1},b_{i+k})
\] 
part~\ref{it:u:6} translates to
\begin{equation}\label{eq:Ii}
	\lambda(I_i)
	> 
	\frac{4}{2k+1}n-2\lambda\big((b_{i+k-1},b_{i+k})\big)\,.
\end{equation}
Furthermore, we note that 
for every $i\in\{2,\dots,k-2\}$ we have 
$
	I_i\subseteq [b_1,b_{k-1}]
$
and each of the two families
\[
	\cI_{0}=\{I_i\colon \text{$i$ even and $2\leq i\leq k-2$} \}
	\qand
	\cI_{1}=\{I_i\colon \text{$i$ odd and $2\leq i\leq k-2$}\}
\]
consists of mutually disjoint intervals. Moreover, we can add the interval~$[b_1,b_2)$ to $\cI_{1}$ 
and~$(b_{k-2},b_{k-1}]$ either to $\cI_{1}$
(when $k$ is even) or to $\cI_{0}$ (when $k$ is odd) and still 
each family consists of mutually disjoint intervals all contained in $[b_1,b_{k-1}]$.
As a result we get 
\[
	2\lambda\big([b_1,b_{k-1}]\big)
	\geq
	\lambda\big([b_1,b_2)\big)+\sum_{i=2}^{k-2}\lambda(I_i)+\lambda\big((b_{k-2},b_{k-1}]\big)
	\,.
\]
Moreover, using the estimate from Lemma~\ref{lem:useful}~\ref{it:u:3} for $\lambda\big([b_1,b_2)\big)$
and $\lambda\big((b_{k-2},b_{k-1}]\big)$ and~\eqref{eq:Ii} for every term in the middle sum,
we arrive at
\begin{align*}
	2\lambda\big([b_1,b_{k-1}]\big)
	&> 
	\big(n-2\alpha(G)\big)
		+\sum_{i=2}^{k-2}\left(\frac{4n}{2k+1}-2\lambda\big((b_{i+k-1},b_{i+k})\big)\right)
		+\big(n-2\alpha(G)\big)\\
	&\geq
	2n-4\alpha(G)+(k-3)\cdot\frac{4n}{2k+1}-2\lambda\big((b_{k+1},b_{2k-2})\big)\\
	&=
	\frac{8k-10}{2k+1}n-4\alpha(G)-2\lambda\big((b_{k+1},b_{2k-2})\big)\,.
\end{align*}
Rearranging the last inequality gives~\eqref{eq:g3} and this concludes the proof.
\end{proof}


\begin{bibdiv}
\begin{biblist}

\bib{A62}{article}{
   author={Andr\'asfai, B.},
   title={\"Uber ein Extremalproblem der Graphentheorie},
   language={German},
   journal={Acta Math. Acad. Sci. Hungar.},
   volume={13},
   date={1962},
   pages={443--455},
   issn={0001-5954},
   review={\MR{0145503}},
}

\bib{A64}{article}{
   author={Andr\'asfai, B.},
   title={Graphentheoretische Extremalprobleme},
   language={German},
   journal={Acta Math. Acad. Sci. Hungar},
   volume={15},
   date={1964},
   pages={413--438},
   issn={0001-5954},
   review={\MR{0169227}},
}

\bib{AnErSo74}{article}{
   author={Andr{\'a}sfai, B.},
   author={Erd{\H{o}}s, P.},
   author={S{\'o}s, V. T.},
   title={On the connection between chromatic number, maximal clique and
   minimal degree of a graph},
   journal={Discrete Math.},
   volume={8},
   date={1974},
   pages={205--218},
   issn={0012-365X},
   review={\MR{0340075}},
}

\bib{ChJiKo97}{article}{
   author={Chen, C. C.},
   author={Jin, G. P.},
   author={Koh, K. M.},
   title={Triangle-free graphs with large degree},
   journal={Combin. Probab. Comput.},
   volume={6},
   date={1997},
   number={4},
   pages={381--396},
   issn={0963-5483},
   review={\MR{1483425}},
   doi={10.1017/S0963548397003167},
}

\bib{Er57}{article}{
   author={Erd\H os, P.},
   title={Remarks on a theorem of {R}amsay},
   journal={Bull. Res. Council Israel. Sect. F},
   volume={7F},
   date={1957/1958},
   pages={21--24},
   review={\MR{0104594}},
}


\bib{Er76a}{article}{
   author={Erd{\H{o}}s, P.},
   title={Problems and results in graph theory and combinatorial analysis},
   conference={
      title={Proceedings of the Fifth British Combinatorial Conference},
      address={Univ. Aberdeen, Aberdeen},
      date={1975},
   },
   book={
      publisher={Utilitas Math., Winnipeg, Man.},
   },
   date={1976},
   pages={169--192. Congressus Numerantium, No. XV},
   review={\MR{0409246}},
}

\bib{Er97a}{article}{
   author={Erd{\H{o}}s, Paul},
   title={Some old and new problems in various branches of combinatorics},
   note={Graphs and combinatorics (Marseille, 1995)},
   journal={Discrete Math.},
   volume={165/166},
   date={1997},
   pages={227--231},
   issn={0012-365X},
   review={\MR{1439273}},
   doi={10.1016/S0012-365X(96)00173-2},
}

\bib{ErFaRoSc94}{article}{
   author={Erd{\H{o}}s, P.},
   author={Faudree, R. J.},
   author={Rousseau, C. C.},
   author={Schelp, R. H.},
   title={A local density condition for triangles},
   note={Graph theory and applications (Hakone, 1990)},
   journal={Discrete Math.},
   volume={127},
   date={1994},
   number={1-3},
   pages={153--161},
   issn={0012-365X},
   review={\MR{1273598}},
   doi={10.1016/0012-365X(92)00474-6},
}

\bib{Ha82}{article}{
   author={H{\"a}ggkvist, Roland},
   title={Odd cycles of specified length in nonbipartite graphs},
   conference={
      title={Graph theory},
      address={Cambridge},
      date={1981},
   },
   book={
      series={North-Holland Math. Stud.},
      volume={62},
      publisher={North-Holland, Amsterdam-New York},
   },
   date={1982},
   pages={89--99},
   review={\MR{671908}},
}

\bib{Ji95}{article}{
   author={Jin, Guo Ping},
   title={Triangle-free four-chromatic graphs},
   journal={Discrete Math.},
   volume={145},
   date={1995},
   number={1-3},
   pages={151--170},
   issn={0012-365X},
   review={\MR{1356592}},
   doi={10.1016/0012-365X(94)00063-O},
}

\bib{KeSu06}{article}{
   author={Keevash, Peter},
   author={Sudakov, Benny},
   title={Sparse halves in triangle-free graphs},
   journal={J. Combin. Theory Ser. B},
   volume={96},
   date={2006},
   number={4},
   pages={614--620},
   issn={0095-8956},
   review={\MR{2232396}},
   doi={10.1016/j.jctb.2005.11.003},
}

\bib{Kr95}{article}{
   author={Krivelevich, Michael},
   title={On the edge distribution in triangle-free graphs},
   journal={J. Combin. Theory Ser. B},
   volume={63},
   date={1995},
   number={2},
   pages={245--260},
   issn={0095-8956},
   review={\MR{1320169}},
   doi={10.1006/jctb.1995.1018},
}

\bib{LS}{article}{
   author={Letzter, S.},
   author={Snyder, R.},   
  title={The homomorphism threshold of $\{C_3,C_5\}$-free graphs},
   eprint={1610.04932},
   note={Submitted.}
}



\bib{NoYe15}{article}{
   author={Norin, Sergey},
   author={Yepremyan, Liana},
   title={Sparse halves in dense triangle-free graphs},
   journal={J. Combin. Theory Ser. B},
   volume={115},
   date={2015},
   pages={1--25},
   issn={0095-8956},
   review={\MR{3383248}},
   doi={10.1016/j.jctb.2015.04.006},
}

\bib{W73}{article}{
   author={Woodall, D. R.},
   title={The binding number of a graph and its Anderson number},
   journal={J. Combinatorial Theory Ser. B},
   volume={15},
   date={1973},
   pages={225--255},
   review={\MR{0327573}},
}

\end{biblist}
\end{bibdiv}
\end{document}